\documentclass{conm-p-l}
\usepackage{amsmath,amssymb,amsthm,amsfonts,amscd,amsopn}
\usepackage{eucal,mathrsfs}
\usepackage{xspace,chngcntr}
\usepackage{rotating,mathtools}
\usepackage[all,cmtip,color]{xy}\xyoption{dvips}

\usepackage{comment} 
\usepackage{supertabular}
\usepackage
{hyperref}
\usepackage[dvipsnames,usenames]{xcolor}
\hypersetup{
    colorlinks=true,
    linkcolor={blue!50!black},
    citecolor={green!50!black},
    urlcolor={red!80!black}
}

\usepackage{datetime,chngcntr}
\usepackage{paralist}

\usepackage{enumitem}
\usepackage{eurosym}

\usepackage{mfirstuc}

   [2017/06/13 v0.1 finetuning mabliautoref and theorem setup]

\RequirePackage{hyperref}

\newcommand{\sectionsize}{\footnotesize}
\newcommand{\theoremsize}{\small}

\renewcommand{\subsectionautorefname}{\sectionsize\sf \subsectionautorefname}

\makeatletter
\@ifdefinable\equationname{\let\equationname\equationautorefname}
\def\equationautorefname~#1\@empty\@empty\null{\protect{\theoremsize\sf
    (#1\@empty\@empty\null)}}%
\@ifdefinable\AMSname{\let\AMSname\AMSautorefname}
\def\AMSautorefname~#1\@empty\@empty\null{\rm (#1\@empty\@empty\null)}%
\@ifdefinable\itemname{\let\itemname\itemautorefname}
\def\itemautorefname~#1\@empty\@empty\null{\theoremsize\sf #1\@empty\@empty\null%
}%
\makeatother

\RequirePackage{amsthm}
\RequirePackage{aliascnt}
\newcommand{\basetheorem}[3]{%
    \newtheorem{#1}{#2}[#3]
    \newtheorem*{#1*}{#2}
    \expandafter\def\csname #1autorefname\endcsname{#2}
}%
\newcommand{\maketheorem}[3]{%
    \newaliascnt{#1}{#2}
    \newtheorem{#1}[#1]{\theoremsize\sf #3}
    \aliascntresetthe{#1}
    \expandafter\def\csname #1autorefname\endcsname{\theoremsize\sf #3}
    \newtheorem{#1*}{#3}
}%
\newcommand{\baseremark}[3]{%
    \newtheorem{#1}{#2}{#3}
    \newtheorem*{#1*}{#2}
    \expandafter\def\csname #1autorefname\endcsname{#2}
}%
\newcommand{\makeremark}[3]{%
    \newaliascnt{#1}{#2}
    \newtheorem{#1}[equation]{#3}
    \aliascntresetthe{#1}
    \expandafter\def\csname #1autorefname\endcsname{\theoremsize\sf #3}
    \newtheorem{#1*}{#3}
}%
\theoremstyle{plain}   

\basetheorem{theorem}{Theorem}{section}

\newcounter{are-there-sections}
\usepackage{amsbsy}

\usepackage{marvosym}

\DeclareMathAlphabet{\smallchanc}{OT1}{pzc}%
                                 {m}{it}
\DeclareFontFamily{OT1}{pzc}{}
\DeclareFontShape{OT1}{pzc}{m}{it}%
             {<-> s * [1.100] pzcmi7t}{}
\DeclareMathAlphabet{\mathchanc}{OT1}{pzc}%
                                 {m}{it}

\newcommand{\mcH}{\mathchanc{H}}

\newcommand{\mcR}{\mathchanc{R}}

\newcommand{\mcm}{\mathchanc{m}}

\newcommand{\mco}{\mathchanc{o}}

\newcommand{\sA}{\mathscr{A}}
\newcommand{\sB}{\mathscr{B}}

\newcommand{\sE}{\mathscr{E}}
\newcommand{\sF}{\mathscr{F}}
\newcommand{\sG}{\mathscr{G}}

\newcommand{\sL}{\mathscr{L}}

\newcommand{\sN}{\mathscr{N}}
\newcommand{\sO}{\mathscr{O}}

\newcommand{\bN}{\mathbb{N}}

\newcommand{\bP}{\mathbb{P}}
\newcommand{\bQ}{\mathbb{Q}}

\newcommand{\bZ}{\mathbb{Z}}

\DeclareSymbolFont{largesymbolsA}{U}{jkpexa}{m}{n}
\SetSymbolFont{largesymbolsA}{bold}{U}{jkpexa}{bx}{n}
\DeclareMathSymbol{\varprod}{\mathop}{largesymbolsA}{16}

\newcommand{\properideal}%
        {\subsetneq}

\newcommand{\wt}{\widetilde}
\newcommand{\what}{\widehat}

\DeclareMathOperator{\coker}{{coker}}

\newcommand{\sHom}[0]{{\mcH\mco\mcm}}

\DeclareMathOperator{\im}{{im}}

\DeclareMathOperator{\Spec}{{Spec}}

\DeclareMathOperator{\sym}{{Sym}}

\DeclareMathOperator{\ver}{{\,\vert\,}}

\newcommand{\factor}[2]{\left. \raise .2em\hbox{\ensuremath{#1}\vphantom{$I^d$}}
\hskip -.1em \right/ \hskip -.4em \raise -.3em\hbox{\ensuremath{#2}}}%
\newcommand\mtimes[3]{{\varprod_{#1}^{#2}}_{\raise 1ex \hbox{\scriptsize #3}}}%

\newcommand{\myR}{{\mcR\!}}

\newcommand{\kdot}{{{\,\begin{picture}(1,1)(-1,-2)\circle*{2}\end{picture}\,}}}

\newcommand{\dcx}[1]{{\omega}^\kdot_{#1}}

\def\dimcoh#1.#2.#3.{h^{#1}(#2,#3)}
\def\hypcoh#1.#2.#3.{\mathbb H_{\vphantom{l}}^{#1}(#2,#3)}
\def\loccoh#1.#2.#3.#4.{H^{#1}_{#2}(#3,#4)}
\def\dimloccoh#1.#2.#3.#4.{h^{#1}_{#2}(#3,#4)}
\def\lochypcoh#1.#2.#3.#4.{\mathbb H^{#1}_{#2}(#3,#4)}
\def\seslong#1.#2.#3.{0  \longrightarrow  #1   \longrightarrow 
 #2 \longrightarrow #3 \longrightarrow 0} 
\def\sesshort#1.#2.#3.{0
 \rightarrow #1 \rightarrow #2 \rightarrow #3 \rightarrow 0}
\def\dist#1.#2.#3.{  #1   \longrightarrow 
 #2 \longrightarrow #3 \stackrel{+1}{\longrightarrow} } 
\def\CDdist#1.#2.#3.{  #1   @>>>  #2  @>>>   #3 @>+1>> }  
\def\shortses#1.#2.#3.{0  \rightarrow  #1   \rightarrow 
 #2  \rightarrow   #3 \rightarrow  0}
\def\shortdist#1.#2.#3.{  #1   \rightarrow 
 #2  \rightarrow   #3 \stackrel{+1}{\rightarrow} }  
\def\ddist#1.#2.#3.#4.#5.#6.{\CD
#1 @>>> #2 @>>> #3 @>+1>> \\
@VVV @VVV @VVV \\
#4 @>>> #5 @>>> #6 @>+1>> 
\endCD}
\def\ddistun#1.#2.#3.#4.#5.#6.{\CD
#1 @>>> #2 @>>> #3 @>+1>> \\
@. @VVV @VVV  \\
#4 @>>> #5 @>>> #6 @>+1>> 
\endCD}
\def\Iff#1#2#3{
\hfil\hbox{\hsize =#1
\vtop{\noin #2}
\hskip.5cm 
\lower.5\baselineskip\hbox{$\Leftrightarrow$}\hskip.5cm
\vtop{\noin #3}}\hfil\medskip}
\newcommand{\union}\cup
\newcommand{\intersect}\cap
\newcommand{\Union}\bigcup
\newcommand{\Intersect}\bigcap
\def\myoplus#1.#2.{\underset #1 \to {\overset #2 \to \oplus}}

\newcommand{\resto}[1]{\raise -.5ex\hbox{$\vert$}_{#1}}

\begin{document}
\makeatletter
\definecolor{brick}{RGB}{204,0,0}
\def\@cite#1#2{{%
 \m@th\upshape\mdseries[{\small\sffamily #1}{\if@tempswa, \small\sffamily
   \color{brick} #2\fi}]}}
\newcommand{\sandor}{{\color{blue}{S\'andor \mdyydate\today}}}
\newenvironment{refmr}{}{}
%
\newcommand\james{M\hskip-.1ex\raise .575ex \hbox{\text{c}}\hskip-.075ex Kernan}

%
\renewcommand\thesubsection{\thesection.\Alph{subsection}}
\renewcommand\subsection{
  \renewcommand{\sfdefault}{phv}
  \@startsection{subsection}%
  {2}{0pt}{-\baselineskip}{.2\baselineskip}{\raggedright
    \sffamily\itshape\small
  }}
\renewcommand\section{
  \renewcommand{\sfdefault}{phv}
  \@startsection{section} %
  {1}{0pt}{\baselineskip}{.2\baselineskip}{\centering
    \sffamily
    \scshape
}}

\setlist[enumerate]{itemsep=3pt,topsep=3pt,leftmargin=2em,label={\rm (\roman*)}}
\newlist{enumalpha}{enumerate}{1}
\setlist[enumalpha]{itemsep=3pt,topsep=3pt,leftmargin=2em,label=(\alph*\/)}
\newcounter{parentthmnumber}
\setcounter{parentthmnumber}{0}
\newcounter{currentparentthmnumber}
\setcounter{currentparentthmnumber}{0}
\newcounter{nexttag}
\newcommand{\setnexttag}{%
  \setcounter{nexttag}{\value{enumi}}%
  \addtocounter{nexttag}{1}%
}
\newcommand{\placenexttag}{%
\tag{\roman{nexttag}}%
}

\newenvironment{thmlista}{%
\label{parentthma}
\begin{enumerate}
}{%
\end{enumerate}
}
\newlist{thmlistaa}{enumerate}{1}
\setlist[thmlistaa]{label=(\arabic*), ref=\autoref{parentthm}\thethm(\arabic*)}
\newcommand*{\parentthmlabeldef}{%
  \expandafter\newcommand
  \csname parentthm\the\value{parentthmnumber}\endcsname
}
\newcommand*{\ptlget}[1]{%
  \romannumeral-`\x
  \ltx@ifundefined{parentthm\number#1}{%
    \ltx@space
    \parentthmundefined
  }{%
    \expandafter\ltx@space
    \csname mymacro\number#1\endcsname
  }%
}  
\newcommand*{\parentthmundefined}{\textbf{??}}
\parentthmlabeldef{parentthm}
\newenvironment{thmlistr}{%
\label{parentthm}
\begin{thmlistrr}}{%
\end{thmlistrr}}
\newlist{thmlistrr}{enumerate}{1}
\setlist[thmlistrr]{label=(\roman*), ref=\autoref{parentthm}(\roman*)}
\newcounter{proofstep}%
\setcounter{proofstep}{0}%
\newcommand{\pstep}[1]{%
  \smallskip
  \noindent
  \emph{{\sc Step \arabic{proofstep}:} #1.}\addtocounter{proofstep}{1}}
\newcounter{lastyear}\setcounter{lastyear}{\the\year}
\addtocounter{lastyear}{-1}
\newcommand\sideremark[1]{%
\normalmarginpar
\marginpar
[
\hskip .45in
\begin{minipage}{.75in}
\tiny #1
\end{minipage}
]
{
\hskip -.075in
\begin{minipage}{.75in}
\tiny #1
\end{minipage}
}}
\newcommand\rsideremark[1]{
\reversemarginpar
\marginpar
[
\hskip .45in
\begin{minipage}{.75in}
\tiny #1
\end{minipage}
]
{
\hskip -.075in
\begin{minipage}{.75in}
\tiny #1
\end{minipage}
}}
\newcommand\Index[1]{{#1}\index{#1}}
\newcommand\inddef[1]{\emph{#1}\index{#1}}
\newcommand\noin{\noindent}
\newcommand\hugeskip{\bigskip\bigskip\bigskip}
\newcommand\smc{\sc}
\newcommand\dsize{\displaystyle}
\newcommand\sh{\subheading}
\newcommand\nl{\newline}
\newcommand\toappear{\rm (to appear)}
\newcommand\mycite[1]{[#1]}
\newcommand\myref[1]{(\ref{#1})}
\newcommand{\parref}[1]{\eqref{\bf #1}}
\newcommand\myli{\hfill\newline\smallskip\noindent{$\bullet$}\quad}
\newcommand\vol[1]{{\bf #1}\ } 
\newcommand\yr[1]{\rm (#1)\ } 
\newcommand\cf{cf.\ \cite}
\newcommand\mycf{cf.\ \mycite}
\newcommand\te{there exist\xspace}
\newcommand\st{such that\xspace}
\newcommand\CM{Cohen-Macaulay\xspace}
\newcommand\myskip{3pt}
\newtheoremstyle{bozont}{3pt}{3pt}%
     {\itshape}
     {}
     {\bfseries}
     {.}
     {.5em}
     {\thmname{#1}\thmnumber{ #2}\thmnote{ #3}}%
\newtheoremstyle{bozont-sub}{3pt}{3pt}%
     {\itshape}
     {}
     {\bfseries}
     {.}
     {.5em}
     {\thmname{#1}\ \arabic{section}.\arabic{thm}.\thmnumber{#2}\thmnote{ \rm #3}}
\newtheoremstyle{bozont-named-thm}{3pt}{3pt}%
     {\itshape}
     {}
     {\bfseries}
     {.}
     {.5em}
     {\thmname{#1}\thmnumber{#2}\thmnote{ #3}}
\newtheoremstyle{bozont-named-bf}{3pt}{3pt}%
     {}
     {}
     {\bfseries}
     {.}
     {.5em}
     {\thmname{#1}\thmnumber{#2}\thmnote{ #3}}
\newtheoremstyle{bozont-named-sf}{3pt}{3pt}%
     {}
     {}
     {\sffamily}
     {.}
     {.5em}
     {\thmname{#1}\thmnumber{#2}\thmnote{ #3}}
\newtheoremstyle{bozont-named-sc}{3pt}{3pt}%
     {}
     {}
     {\scshape}
     {.}
     {.5em}
     {\thmname{#1}\thmnumber{#2}\thmnote{ #3}}
\newtheoremstyle{bozont-named-it}{3pt}{3pt}%
     {}
     {}
     {\itshape}
     {.}
     {.5em}
     {\thmname{#1}\thmnumber{#2}\thmnote{ #3}}
\newtheoremstyle{bozont-sf}{3pt}{3pt}%
     {}
     {}
     {\sffamily}
     {.}
     {.5em}
     {\thmname{#1}\thmnumber{ #2}\thmnote{ \rm #3}}
\newtheoremstyle{bozont-sc}{3pt}{3pt}%
     {}
     {}
     {\scshape}
     {.}
     {.5em}
     {\thmname{#1}\thmnumber{ #2}\thmnote{ \rm #3}}
\newtheoremstyle{bozont-remark}{3pt}{3pt}%
     {}
     {}
     {\scshape}
     {.}
     {.5em}
     {\thmname{#1}\thmnumber{ #2}\thmnote{ \rm #3}}
\newtheoremstyle{bozont-subremark}{3pt}{3pt}%
     {}
     {}
     {\scshape}
     {.}
     {.5em}
     {\thmname{#1}\ \arabic{section}.\arabic{thm}.\thmnumber{#2}\thmnote{ \rm #3}}
\newtheoremstyle{bozont-def}{3pt}{3pt}%
     {}
     {}
     {\bfseries}
     {.}
     {.5em}
     {\thmname{#1}\thmnumber{ #2}\thmnote{ \rm #3}}
\newtheoremstyle{bozont-reverse}{3pt}{3pt}%
     {\itshape}
     {}
     {\bfseries}
     {.}
     {.5em}
     {\thmnumber{#2.}\thmname{ #1}\thmnote{ \rm #3}}
\newtheoremstyle{bozont-reverse-sc}{3pt}{3pt}%
     {\itshape}
     {}
     {\scshape}
     {.}
     {.5em}
     {\thmnumber{#2.}\thmname{ #1}\thmnote{ \rm #3}}
\newtheoremstyle{bozont-reverse-sf}{3pt}{3pt}%
     {\itshape}
     {}
     {\sffamily}
     {.}
     {.5em}
     {\thmnumber{#2.}\thmname{ #1}\thmnote{ \rm #3}}
\newtheoremstyle{bozont-remark-reverse}{3pt}{3pt}%
     {}
     {}
     {\sc}
     {.}
     {.5em}
     {\thmnumber{#2.}\thmname{ #1}\thmnote{ \rm #3}}
\newtheoremstyle{bozont-def-reverse}{3pt}{3pt}%
     {}
     {}
     {\bfseries}
     {.}
     {.5em}
     {\thmnumber{#2.}\thmname{ #1}\thmnote{ \rm #3}}
\newtheoremstyle{bozont-def-newnum-reverse}{3pt}{3pt}%
     {}
     {}
     {\bfseries}
     {}
     {.5em}
     {\thmnumber{#2.}\thmname{ #1}\thmnote{ \rm #3}}
\newtheoremstyle{bozont-def-newnum-reverse-plain}{3pt}{3pt}%
   {}
   {}
   {}
   {}
   {.5em}
   {\thmnumber{\!(#2)}\thmname{ #1}\thmnote{ \rm #3}}
\newtheoremstyle{bozont-number}{3pt}{3pt}%
   {}
   {}
   {}
   {}
   {0pt}
   {\thmnumber{\!(#2)} }
\newtheoremstyle{bozont-step}{3pt}{3pt}%
   {\itshape}
   {}
   {\scshape}
   {}
   {.5em}
   {$\boxed{\text{\sc \thmname{#1}~\thmnumber{#2}:\!}}$}
\theoremstyle{bozont}    
\ifnum \value{are-there-sections}=0 {%
  \basetheorem{proclaim}{Theorem}{}
} 
\else {%
  \basetheorem{proclaim}{Theorem}{section}
} 
\fi
\maketheorem{thm}{proclaim}{Theorem}
\maketheorem{mainthm}{proclaim}{Main Theorem}
\maketheorem{cor}{proclaim}{Corollary} 
\maketheorem{cors}{proclaim}{Corollaries} 
\maketheorem{lem}{proclaim}{Lemma} 
\maketheorem{prop}{proclaim}{Proposition} 
\maketheorem{conj}{proclaim}{Conjecture}
\basetheorem{subproclaim}{Theorem}{proclaim}
\maketheorem{sublemma}{subproclaim}{Lemma}
\newenvironment{sublem}{%
\setcounter{sublemma}{\value{equation}}
\begin{sublemma}}
{\end{sublemma}}
\theoremstyle{bozont-sub}
\maketheorem{subthm}{equation}{Theorem}
\maketheorem{subcor}{equation}{Corollary} 
\maketheorem{subprop}{equation}{Proposition} 
\maketheorem{subconj}{equation}{Conjecture}
\theoremstyle{bozont-named-thm}
\maketheorem{namedthm}{proclaim}{}
\theoremstyle{bozont-sc}
\newtheorem{proclaim-special}[proclaim]{\specialthmname}
\newenvironment{proclaimspecial}[1]
     {\def\specialthmname{#1}\begin{proclaim-special}}
     {\end{proclaim-special}}
\theoremstyle{bozont-subremark}
\maketheorem{subrem}{equation}{Remark}
\maketheorem{subnotation}{equation}{Notation} 
\maketheorem{subassume}{equation}{Assumptions} 
\maketheorem{subobs}{equation}{Observation} 
\maketheorem{subexample}{equation}{Example} 
\maketheorem{subex}{equation}{Exercise} 
\maketheorem{inclaim}{equation}{Claim} 
\maketheorem{subquestion}{equation}{Question}
\theoremstyle{bozont-remark}
\basetheorem{subremark}{Remark}{proclaim}
\makeremark{subclaim}{subremark}{Claim}
\maketheorem{rem}{proclaim}{Remark}
\maketheorem{claim}{proclaim}{Claim} 
\maketheorem{notation}{proclaim}{Notation} 
\maketheorem{assume}{proclaim}{Assumptions} 
\maketheorem{obs}{proclaim}{Observation} 
\maketheorem{example}{proclaim}{Example} 
\maketheorem{examples}{proclaim}{Examples} 
\maketheorem{complem}{equation}{Complement}
\maketheorem{const}{proclaim}{Construction}   
\maketheorem{ex}{proclaim}{Exercise} 
\newtheorem{case}{Case} 
\newtheorem{subcase}{Subcase}   
\newtheorem{step}{Step}
\newtheorem{approach}{Approach}
\maketheorem{Fact}{proclaim}{Fact}
\newtheorem{fact}{Fact}
\newtheorem*{SubHeading*}{\SubHeadingName}%
\newtheorem{SubHeading}[proclaim]{\SubHeadingName}
\newtheorem{sSubHeading}[equation]{\sSubHeadingName}
\newenvironment{demo}[1] {\def\SubHeadingName{#1}\begin{SubHeading}}
  {\end{SubHeading}}%
\newenvironment{subdemo}[1]{\def\sSubHeadingName{#1}\begin{sSubHeading}}
  {\end{sSubHeading}} %
\newenvironment{demo-r}[1]{\def\SubHeadingName{#1}\begin{SubHeading-r}}
  {\end{SubHeading-r}}%
\newenvironment{subdemo-r}[1]{\def\sSubHeadingName{#1}\begin{sSubHeading-r}}
  {\end{sSubHeading-r}} %
\newenvironment{demo*}[1]{\def\SubHeadingName{#1}\begin{SubHeading*}}
  {\end{SubHeading*}}%
\maketheorem{defini}{proclaim}{Definition}
\maketheorem{question}{proclaim}{Question}
\maketheorem{crit}{proclaim}{Criterion}
\maketheorem{pitfall}{proclaim}{Pitfall}
\maketheorem{addition}{proclaim}{Addition}
\maketheorem{principle}{proclaim}{Principle} 
\maketheorem{condition}{proclaim}{Condition}
\maketheorem{exmp}{proclaim}{Example}
\maketheorem{hint}{proclaim}{Hint}
\maketheorem{exrc}{proclaim}{Exercise}
\maketheorem{prob}{proclaim}{Problem}
\maketheorem{ques}{proclaim}{Question}    
\maketheorem{alg}{proclaim}{Algorithm}
\maketheorem{remk}{proclaim}{Remark}          
\maketheorem{note}{proclaim}{Note}            
\maketheorem{summ}{proclaim}{Summary}         
\maketheorem{notationk}{proclaim}{Notation}   
\maketheorem{warning}{proclaim}{Warning}  
\maketheorem{defn-thm}{proclaim}{Definition--Theorem}  
\maketheorem{convention}{proclaim}{Convention}  
\maketheorem{hw}{proclaim}{Homework}
\maketheorem{hws}{proclaim}{\protect{${\mathbb\star}$}Homework}
\newtheorem*{ack}{Acknowledgment}
\newtheorem*{acks}{Acknowledgments}
\theoremstyle{bozont-number}
\newtheorem{say}[proclaim]{}
\newtheorem{subsay}[equation]{}
\theoremstyle{bozont-def}    
\maketheorem{defn}{proclaim}{Definition}
\maketheorem{subdefn}{equation}{Definition}
\theoremstyle{bozont-reverse}    
\maketheorem{corr}{proclaim}{Corollary} 
\maketheorem{lemr}{proclaim}{Lemma} 
\maketheorem{propr}{proclaim}{Proposition} 
\maketheorem{conjr}{proclaim}{Conjecture}
\theoremstyle{bozont-remark-reverse}
\newtheorem{SubHeading-r}[proclaim]{\SubHeadingName}
\newtheorem{sSubHeading-r}[equation]{\sSubHeadingName}
\newtheorem{SubHeadingr}[proclaim]{\SubHeadingName}
\newtheorem{sSubHeadingr}[equation]{\sSubHeadingName}
\theoremstyle{bozont-reverse-sc}
\newtheorem{proclaimr-special}[proclaim]{\specialthmname}
\newenvironment{proclaimspecialr}[1]%
{\def\specialthmname{#1}\begin{proclaimr-special}}%
{\end{proclaimr-special}}
\theoremstyle{bozont-remark-reverse}
\maketheorem{remr}{proclaim}{Remark}
\maketheorem{subremr}{equation}{Remark}
\maketheorem{notationr}{proclaim}{Notation} 
\maketheorem{assumer}{proclaim}{Assumptions} 
\maketheorem{obsr}{proclaim}{Observation} 
\maketheorem{exampler}{proclaim}{Example} 
\maketheorem{exr}{proclaim}{Exercise} 
\maketheorem{claimr}{proclaim}{Claim} 
\maketheorem{inclaimr}{equation}{Claim} 
\maketheorem{definir}{proclaim}{Definition}
\theoremstyle{bozont-def-newnum-reverse}    
\maketheorem{newnumr}{proclaim}{}
\theoremstyle{bozont-def-newnum-reverse-plain}
\maketheorem{newnumrp}{proclaim}{}
\theoremstyle{bozont-def-reverse}    
\maketheorem{defnr}{proclaim}{Definition}
\maketheorem{questionr}{proclaim}{Question}
\newtheorem{newnumspecial}[proclaim]{\specialnewnumname}
\newenvironment{newnum}[1]{\def\specialnewnumname{#1}\begin{newnumspecial}}{\end{newnumspecial}}
\theoremstyle{bozont-step}
\newtheorem{bstep}{Step}
\newcounter{thisthm} 
\newcounter{thissection} 
\newcommand{\ilabel}[1]{%
  \newcounter{#1}%
  \setcounter{thissection}{\value{section}}%
  \setcounter{thisthm}{\value{proclaim}}%
  \label{#1}}
\newcommand{\iref}[1]{%
  (\the\value{thissection}.\the\value{thisthm}.\ref{#1})}
\newcounter{lect}
\setcounter{lect}{1}
\newcommand\lecture{\newpage\centerline{\sfbf Lecture \arabic{lect}}
\addtocounter{lect}{1}}
\newcounter{topic}
\setcounter{topic}{1}
\newenvironment{topic}
{\noindent{\sc Topic 
\arabic{topic}:\ }}{\addtocounter{topic}{1}\par}
\counterwithin{equation}{proclaim}
\counterwithin{figure}{section} 
\newcommand\equinsect{\numberwithin{equation}{section}}
\newcommand\equinthm{\numberwithin{equation}{proclaim}}
\newcommand\figinthm{\numberwithin{figure}{proclaim}}
\newcommand\figinsect{\numberwithin{figure}{section}}
\newenvironment{sequation}{%
\setcounter{equation}{\value{thm}}
\numberwithin{equation}{section}%
\begin{equation}%
}{%
\end{equation}%
\numberwithin{equation}{proclaim}%
\addtocounter{proclaim}{1}%
}
\newcommand{\num}{\arabic{section}.\arabic{proclaim}}
\newenvironment{pf}{\smallskip \noindent {\sc Proof. }}{\qed\smallskip}
\newenvironment{enumerate-p}{
  \begin{enumerate}}
  {\setcounter{equation}{\value{enumi}}\end{enumerate}}
\newenvironment{enumerate-cont}{
  \begin{enumerate}
    {\setcounter{enumi}{\value{equation}}}}
  {\setcounter{equation}{\value{enumi}}
  \end{enumerate}}
\let\lenumi\labelenumi
\newcommand{\rmlabels}{\renewcommand{\labelenumi}{\rm \lenumi}}
\newcommand{\rmlabelsoff}{\renewcommand{\labelenumi}{\lenumi}}
\newenvironment{heading}{\begin{center} \sc}{\end{center}}
\newcommand\subheading[1]{\smallskip\noindent{{\bf #1.}\ }}
\newlength{\swidth}
\setlength{\swidth}{\textwidth}
\addtolength{\swidth}{-,5\parindent}
\newenvironment{narrow}{
  \medskip\noindent\hfill\begin{minipage}{\swidth}}
  {\end{minipage}\medskip}
\newcommand\nospace{\hskip-.45ex}
\newcommand{\sfbf}{\sffamily\bfseries}
\newcommand{\sfbfs}{\sffamily\bfseries\small}
\newcommand{\twidle}{\textasciitilde}
\makeatother

\newcounter{stepp}
\setcounter{stepp}{0}
\newcommand{\nextstep}[1]{%
  \addtocounter{stepp}{1}%
  \begin{bstep}%
    {#1}
  \end{bstep}%
  \noindent%
}
\newcommand{\resetsteps}{\setcounter{stepp}{0}}

\title
{Non-Cohen-Macaulay canonical singularities} 
\author{S\'andor J Kov\'acs} 
\date{\usdate\today} 
\thanks{Supported in part by NSF Grant 
DMS-1565352 and the Craig McKibben and Sarah Merner Endowed Professorship in
Mathematics at the University of Washington.} 
\address{University of Washington, Department of Mathematics, 
Seattle, WA 98195, USA}
\email{skovacs@uw.edu} 
\urladdr{http://www.math.washington.edu/$\sim$kovacs}
\maketitle
\setcounter{tocdepth}{1}
\vskip-1em
\centerline{\it Dedicated to Lawrence Ein on the occasion of his 60th birthday}

\numberwithin{equation}{section}

\section{Introduction}
\noindent
Since the first counter-example to Kodaira vanishing in positive characteristic was
constructed by Raynaud \cite{MR541027} many other counter-examples have been found
satisfying various prescribed properties
\cite{MR894379,MR972344,MR1128214,Kollar96,MR1385284,MR3079312,dCF15,CT16b,CT16}.  An
elementary counter-example for which the line bundle violating Kodaira vanishing is
very ample was constructed by Lauritzen and Rao in \cite{LR97}. Let us denote it by
$X$. It is straightforward from the construction that $X$ is a rational variety and
for $p=2$ and $\dim X=6$ it is Fano.  Let $Z$ denote the cone over $X$ using the
embedding given by the global sections of the very ample line bundle violating
Kodaira vanishing.  It is well-known that a cone over a Fano variety has klt
singularities if $K_Z$ is $\bQ$-Cartier. (See \autoref{defs}.)
The failure of Kodaira vanishing on $X$ implies that $Z$ will not have \CM
singularities, in particular it does not have rational singularities.
%
As pointed out by 
Esnault and 
Koll\'ar, although in this example $K_Z$ is not $\bQ$-Cartier, one can easily find a
boundary $\Delta$ on $Z$ that makes $K_Z+\Delta$ $\bQ$-Cartier, and hence the pair
$(Z,\Delta)$ klt.  In other words Lauritzen and Rao's counter-example to Kodaira
vanishing produces a klt pair $(Z,\Delta)$ such that $Z$ is not \CM. This provides a
counter-example to the positive characteristic analogue of Elkik's theorem
\cite{Elkik81}, \cite[5.22]{KM98}.
Examples of non-\CM klt singularities were also given by Yasuda in \cite{MR3230848}
and Cascini and Tanaka in \cite{CT16b}.

We will show that one can use the above $X$ to produce even more interesting
singularities.  I will demonstrate below that in fact the very ample line bundle
$\omega_X^{-2}$ also violates Kodaira vanishing and hence leads to a cone, using the
polarization given by $\omega_X^{-1}$, whose canonical sheaf is a line bundle, has
canonical singularities, and is not \CM. Of course, then it also does not have
rational singularities. In other words, the purpose of this note is to prove the
following.

\begin{thm}\label{thm:main1}
  Let $k$ be a field of characteristic $2$. Then there exists a Fano variety $X$ over
  $k$ such that
  \begin{enumerate}
  \item $\dim X=6$,
  \item $\omega_X^{-1}$ is very ample, and
  \item $\omega_X^{-2}$ violates Kodaira vanishing.
  \end{enumerate}
\end{thm}

One might ask if there is a similar example in smaller dimensions. It follows from
\cite[A.1]{CT16} that there are no Del Pezzo surfaces with this property. Hence a
similar example in smaller dimension would be at least $3$-dimensional.
One might also ask if there is a similar example where $\omega_X^{-1}$ violates
Kodaira vanishing. The example here is certainly not such and it is well-known that
no such example exists for $\dim X=2,3$ \cite{MR1440723,MR2347424,MR3493588}. While
this is an interesting question, it is irrelevant for the purposes of the present
article.
The more interesting question is whether there are similar examples in all positive
characteristics. 

My main interest in the above result lies in the following application.  By taking
the cone over $X$ given by the embedding induced by the global sections of
$\omega_X^{-1}$ we obtain the following.

\begin{thm}\label{thm:main2}
  Let $k$ be a field of characteristic $2$. Then there exists a variety $Z$ over $k$
  with the following properties:
  \begin{enumalpha}
  \item $Z$ is of dimension $7$, 
    has a single isolated canonical singularity, and
    admits a resolution of singularities by a smooth variety over $k$,
  \item $\omega_Z$ is a line bundle,
  \item $Z$ is not \CM, in particular, $Z$ is not Gorenstein and does not have
    rational singularities.
  \end{enumalpha}
\end{thm}

Again, one might ask if there are such singularities in smaller dimensions. Of
course, if one finds examples such as in \autoref{thm:main1} in smaller dimensions,
that would provide smaller dimensional examples for \autoref{thm:main2} as
well. However, as mentioned above there are no examples similar to
\autoref{thm:main1} in dimension $2$ which makes it an interesting question whether
there exist $3$-dimensional canonical singularities, perhaps even of index $1$, that
are not \CM.
And again, the possibly more interesting question is whether there are such
singularities in all positive characteristics.  

\noindent{\smc Note added in proof.}
While this paper was under review the above question has been answered.  Bernasconi
gave examples of non-\CM klt singularities in characteristic $3$ in \cite{Ber17} and
Totaro and Yasuda gave examples of non-\CM terminal singularities in all positive
characteristics in \cite{Tot17} and \cite{Yas17}.

In the opposite direction Hacon and Witaszek \cite{HW17} recently proved that in
dimension $3$ klt singularities are rational if the characteristic of the base field
is sufficiently large.

\ack{I am grateful to J\'anos Koll\'ar, Hiromu Tanaka, Burt Totaro, Takehiko Yasuda,
  and to the referee for useful comments.}

\section{Non-\CM singularities via failure of Kodaira vanishing}
\noindent
%
\begin{defini}\label{defs}
  Let $X$ be a smooth projective variety over $k$ and $\sL$ an ample line bundle on
  $X$. Then we will say that $\sL$ \emph{violates Kodaira vanishing} if there exists
  an $i<\dim X$ such that $H^i(X,\sL^{-1})\neq 0$.  By Serre duality this is
  equivalent to that $H^{\dim X-i}(X,\sL\otimes\omega_X)\neq 0$.

  The canonical divisor of a normal variety $Z$ is denoted, as usual, by $K_Z$ and
  the associated reflexive sheaf of rank $1$, the canonical sheaf, is denoted by
  $\omega_Z$. I.e., $\omega_Z\simeq \sO_Z(K_Z)$. A Weil divisor $D$ on $Z$ is
  $\bQ$-Cartier if there exists a non-zero $m\in \bN$ such that $mD$ is Cartier.
  A normal variety $Z$ is said to have \emph{rational singularities} if for a
  resolution of singularities $\phi:\wt Z\to Z$ the following conditions hold:
  \begin{enumerate}
  \item\label{item:4} $\myR^i\phi_*\sO_{\wt Z}=0$ for $i>0$, and
  \item\label{item:5} $\myR^if_*\omega_{\wt Z}=0$ for $i>0$.
  \end{enumerate} 
  In characteristic $0$ \autoref{item:5} is automatic by the Grauert-Riemenschneider
  vanishing theorem \cite{Gra-Rie70b}, \cite[2.68]{KM98}.
  For the definition of \emph{klt} and \emph{canonical} singularities the reader is
  referred to \cite[2.8]{SingBook}.
\end{defini}

Rational singularities are \CM by the following well-known lemma. 
 A very short proof is included for the convenience of the reader.

\begin{lem}\label{lem:rtl-is-CM}
  Let $Z$ be a scheme with rational singularities. Then $Z$ is \CM.
\end{lem}

\begin{proof}
  Let $d=\dim Z$ and let $\phi:\wt Z\to Z$ be a resolution of singularities of
  $Z$. This implies that $\sO_Z\simeq \myR\phi_*\sO_{\wt Z}$ and $\omega_Z\simeq
  \myR\phi_*\omega_{\wt Z}$. Then by Grothendieck duality
  \[
  \omega_Z[d]\simeq %
  \myR\phi_*\omega_{\wt Z}[d]\simeq %
  \myR\phi_*\myR\sHom_{\wt Z}(\sO_{\wt Z}, \dcx {\wt Z}) \simeq %
  \myR\sHom_{Z}(\myR\phi_*\sO_{\wt Z}, \dcx {Z})\simeq %
  \dcx {Z},
  \]
  and hence $Z$ is \CM.
\end{proof}

\begin{rem}
  It follows easily that if $Z$ is not \CM, then for any resolution of singularities
  $\phi:\wt Z\to Z$, there exists an $i>0$ such that either $\myR^i\phi_*\sO_{\wt
    Z}\neq 0$ or $\myR^i\phi_*\omega_{\wt Z}\neq 0$.
\end{rem}

\noindent
Next I will review the more-or-less well-known idea of constructing non-\CM
singularities as cones over varieties violating Kodaira vanishing.

\begin{say}\label{say:construction}
  Let $X$ be a normal projective variety over a field $k$ of characteristic $p>0$,
  $\sL$ an ample line bundle on $X$, and
  $Z=C_a(X,\sL)=\Spec\left(\oplus_{m\geq 0} H^0(X, \sL^m)\right)$ the \emph{affine
    cone} over $X$ with conormal bundle $\sL$. (Here we follow the convention of
  \cite[3.8]{SingBook} on cones.)
\end{say}

\noindent
Then we have the following well-known criterion
cf.~\cite[3.11]{SingBook}: 
\begin{say}\label{say:CM}
  $Z$ is \CM if and only if $H^i(X,\sL^{q})=0$ for all $0<i<\dim X$ and $q\in \bZ$.
\end{say}

\noindent
This implies for example that cones over varieties whose structure sheaves have
non-trivial \emph{middle} cohomology, for instance abelian varieties of dimension at
least $2$, are not \CM. 
It also implies that
\begin{say}\label{say:not-CM}
  if some power of $\sL$ violates Kodaira vanishing, then $Z$ is not \CM.
\end{say}

\noindent
Next recall that the canonical divisor of a canonical singularity is $\bQ$-Cartier
and observe that in the above construction
\begin{say}\label{say:Q-Gor}
  if $\omega_X^{r}\simeq \sL^{q}$ for some $r,q\in\bZ$, $r\neq 0$, then $K_Z$ is
  $\bQ$-Cartier of index at most $r$.
\end{say}

However, even if \eqref{say:Q-Gor} fails, $Z$ may still provide an example of a klt
singularity with an appropriate boundary as we will see in the next statement, which
summarizes what we found in this section.
Note that this statement is a simple consequence of the combination of \cite[3.1,
3.11]{SingBook}.

\begin{prop}\label{prop:klt}
  In addition to the definitions in \eqref{say:construction} assume that
  $X$ is a smooth Fano variety
  and that some power of $\sL$ violates Kodaira vanishing.  Then there exists a
  $\bQ$-divisor $\Delta$ on $Z$ such that
  \begin{enumerate}
  \item\label{item:2} $(Z,\Delta)$ has klt singularities,
  \item\label{item:1} $Z$ is not \CM, and hence in particular has non-rational
    singularities, and 
  \item\label{item:3} if $\omega_X\simeq \sL^q$ for some $q\in\bZ$, then $Z$ has
    canonical singularities.
  \end{enumerate}
\end{prop}

\begin{proof}
  Since $\omega_X^{-1}$ is ample, there is an $r\in\bN$, $r>0$, such that
  $\sN=\sL^{-1}\otimes \omega_X^{-r}$ is also ample.  Let $N$ be a general member of
  the complete linear system corresponding to $\sN^{m}$ for some $m\gg 0$,
  $\what N\subseteq Z$ the cone over $N$, and $\Delta:=\frac 1{rm} \what N$.  Then
  $\sO_X(rm(K_X+\frac 1{rm}N))\simeq \sL^{-m}$, so $K_Z+\Delta$ is a Cartier divisor
  on $Z$ (cf.\cite[3.14(4)]{SingBook}), and hence $K_Z+\Delta$ is
  $\bQ$-Cartier. Furthermore, $N$ is smooth and hence $(X,\frac 1{rm} N)$ is klt, so
  \autoref{item:2} follows from \cite[3.1(3)]{SingBook}.
  Now, if $\omega_X\simeq \sL^q$ for some $q\in\bZ$, then $\omega_Z$ is a line bundle
  and hence \autoref{item:3} follows from \autoref{item:2}. 
  Finally, \autoref{item:1} is simply a restatement of \eqref{say:not-CM}. 
\end{proof}

\section{The construction of Lauritzen and Rao}\label{sec:constr-laur-rao}

\noindent
Next, I will recall the construction of Lauritzen and Rao from \cite{LR97}.

Let $V$ be a vector space of dimension $n+1$ over a field $k$ of characteristic $p$
where $p\geq n-1\geq 2$, and let $\bP(V)\simeq \bP^n$ be the associated projective
space of dimension $n$. 
Let $W:=\bP(V)\times \bP(V^\vee)$ and for $a,b\in\bZ$ let $\sO_W(a,b)$ denote the
line bundle $\sO_{\bP(V)}(a)\boxtimes \sO_{\bP(V^\vee)}(b)$ on $W$.
Next let $\sA$ be the locally free sheaf defined by the short exact sequence
\begin{equation}
  \label{eq:9}
  \xymatrix{%
    0\ar[r] & \sA \ar[r] & V\otimes \sO_{\bP(V)} \ar[r] & \sO_{\bP(V)}(1) \ar[r] & 0, 
  }
\end{equation}
and let $\alpha: Y:=\bP(\sA^\vee)\to \bP(V)$ be the projective space bundle over
$\bP(V)$ associated to $\sA^\vee$. Let $\sO_{\alpha}(1)$ denote the corresponding
tautological line bundle on $Y$. Then there exists another associated short exact
sequence on $Y$:
\begin{equation}
  \label{eq:3}
  \xymatrix{%
    0\ar[r] & \sG \ar[r] & \alpha^*\sA^\vee \ar[r] & \sO_{\alpha}(1) \ar[r] & 0,  }
\end{equation}
which defines the locally free sheaf $\sG$ on $Y$. It is shown in \cite[p.23]{LR97}
that $Y$ admits a closed embedding into $W\simeq \bP^n\times {\bP}^n$ with
bihomogenous coordinate ring
\begin{equation}
  \label{eq:10}
  \factor{k[x_0,\dots,x_n;y_0,\dots,y_n]}{(\sum x_iy_i )}.
\end{equation}
In particular, the ideal sheaf of $Y$ in $W$ is $\sO_W(-1,-1)$.  Let
$\sO_Y(a,b):=\sO_W(a,b)\resto Y$.  Then it follows easily that
\begin{equation}
  \label{eq:8}
  \omega_Y\simeq \sO_Y(-n,-n),\quad \alpha^*\sO_{\bP(V)}(a)\simeq \sO_Y(a,0),
  \quad\text{ and }\quad \sO_\alpha(b) \simeq \sO_Y(0,b).
\end{equation}
Let $\eta$ be defined as the composition of the natural morphisms induced by the
morphisms in \autoref{eq:9} and \autoref{eq:3} using the isomorphisms in
\autoref{eq:8}:
\begin{equation}
  \label{eq:11}
  \xymatrix{%
    V^\vee\otimes\sO_Y \ar[r]\ar@/^1.5em/[rr]^\eta & \alpha^*\sA^\vee \ar[r] & \sO_Y(0,1)
  }
\end{equation}
Then we have the following commutative diagram, where $\sB=\ker\eta$:
\begin{equation}
  \label{eq:6}
  \begin{aligned}
    \xymatrix
    {%
      & 0 \ar[d] & 0\ar[d] \\
      & \sO_Y(-1,0) \ar[d] \ar[r]^=     & \sO_Y(-1,0) \ar[d] \\
      0 \ar[r] & \sB \ar[d]_\tau\ar[r] & V^\vee\otimes\sO_Y \ar[d]\ar[r]^\eta &
      \sO_Y(0,1) \ar[d]^=
      \ar[r] & 0 \\
      0 \ar[r] & \sG \ar[d]\ar[r] & \alpha^*\sA^\vee \ar[d]\ar[r] & \sO_Y(0,1)
      \ar[r] & 0 \\
      & 0 & 0 }
  \end{aligned}
\end{equation}

\noindent
Finally, let $\sG'=\sG\otimes \sO_\alpha(1)$ and $\pi: X:=\bP(F^*\sG')\to Y$, where
$F:Y\to Y$ is the absolute Frobenius morphism of $Y$. Note that by construction $\dim
X=3n-3$ and $\dim Y=2n-1$.

\begin{subsay}\label{say:very}
  Again, it is shown in \cite[p.23]{LR97} that the tautological line bundle of $\pi$,
  denoted by $\sO_\pi(1)$, is globally generated and the line bundle
  $\pi^*\sO_Y(1,1)\otimes\sO_\pi(1)$ is very ample. It follows that
  $\pi^*\sO_Y(1,1)\otimes\sO_\pi(q)$ is also very ample for any $q>0$.
\end{subsay}

Using the formula for the canonical bundle of a projective space bundle, one obtains
that 
\begin{equation}
  \label{eq:1}
  \omega_X\simeq \pi^*\sO_Y(p-n, p(n-2)-n)\otimes\sO_\pi(-n+1)
\end{equation}

As it was pointed out by H\'el\`ene Esnault if one chooses the values $p=2$ and
$n=3$, then $X$ is a Fano variety and hence there exists a klt pair $(Z,\Delta)$
where $Z$ is not \CM, in particular, it does not have rational singularities
cf.~\autoref{prop:klt}.

\section{A Fano variety 
  violating Kodaira vanishing}

\noindent
We will use the above construction and prove that if $p=2$ and $n=3$, then the very
ample line bundle $\omega_X^{-2}$ violates Kodaira vanishing. To do this, first we
need to compute a few auxiliary cohomology groups.
\noindent
We will keep using the notation introduced in \autoref{sec:constr-laur-rao}.

\setcounter{proclaim}{\value{equation}}
\numberwithin{equation}{proclaim}
\begin{lem}\label{lem:a-b}
  Let $a,b\in\bZ$. 
  Then 
  \[
  H^i(Y,\sO_Y(a,b))=0
  \]
  if either 
  \begin{enumerate}
  \item\label{item:6} $a$ and $b$ are arbitrary and $0<i<n-1$, or
  \item\label{item:7} $a,b>-n$ and $i>0$, or
  \item\label{item:8} at least one of $a$ and $b$ is negative and $i=0$.
  \end{enumerate}
\end{lem}

\begin{proof}
  By (\ref{eq:10}) we have the following short exact sequence:
  \[
  \xymatrix{%
    0\ar[r] & \sO_W(a-1,b-1) \ar[r] & \sO_W(a,b) \ar[r] & \sO_Y(a,b) \ar[r] & 0
  }
  \]
  Since $W\simeq \bP^n\times \bP^n$, using the K\"unneth formula, the first two
  (non-zero) sheaves above have no cohomology in the following cases:
  \begin{enumalpha}
  \item\label{item:9} for $0<i<n$ and arbitrary $a$ and $b$,
  \item\label{item:10} for $a,b>-n$ and $i>0$, and 
  \item\label{item:11} at least one of $a$ and $b$ is negative and $i=0$.
  \end{enumalpha}
  Then \ref{item:9} implies \autoref{item:6}, \ref{item:10} implies
  \autoref{item:7}, and \ref{item:9} and \ref{item:11} together imply
  \autoref{item:8}.
\end{proof}

\begin{cor}\label{cor:a-b}
  Under the same conditions as in \autoref{lem:a-b},
  \[
  H^i(Y,\sO_Y(a,b)\otimes F^*(V^\vee\otimes\sO_Y) )=0.
  \]
\end{cor}

\begin{proof}
  $F^*(V^\vee\otimes\sO_Y)$ is a free $\sO_Y$ sheaf, so this is straightforward from
  \autoref{lem:a-b}. 
\end{proof}

\begin{lem}\label{lem:B-coker}
  Let $a,b\in\bZ$. 
  Then
  \begin{multline*}
    H^1(Y,\sO_Y(a,b)\otimes F^*\sB)\simeq \\ \simeq \xymatrix@R=3em{%
      \coker\left[ H^0(Y,\sO_Y(a,b)\otimes F^*(V^\vee\otimes\sO_Y) )
        \ar[rr]^-{\eta_1:=F^*\eta} \right.  & & \left. H^0(Y,\sO_Y(a,b+p)) \right] }
  \end{multline*}
  where $\eta_1=F^*\eta$ is induced by the morphism $\eta$ defined in (\ref{eq:11}).
  In particular, if either $a<0$ or $b<-p$, then 
  \[
    H^1(Y,\sO_Y(a,b)\otimes F^*\sB)=0.
  \]
\end{lem}

\begin{proof}
  Consider the Frobenius pull-back of the middle row of the diagram in (\ref{eq:6})
  twisted with $\sO_Y(a,b)$:
  \begin{equation*}
      \xymatrix@C2em{%
        0 \ar[r] & \sO_Y(a,b)\otimes F^*\sB \ar[r] & \sO_Y(a,b)\otimes
        F^*(V^\vee\otimes\sO_Y) \ar[r] & \sO_Y(a, b+p) \ar[r] & 0.}
  \end{equation*}
  Then, since $n>2$, both statements follow from \autoref{cor:a-b}.
\end{proof}

\begin{lem}\label{lem:B-G}
  Let $a,b\in\bZ$. 
  Then the morphism induced by $\tau$ in (\ref{eq:6}) is an isomorphism:
  \begin{equation*}
    \xymatrix@R=.1em{%
      H^1(\sO_Y(a,b)\otimes F^*\sB)
      \ar[r]^\simeq & H^1(\sO_Y(a,b)\otimes F^*\sG). }    
  \end{equation*}
  Furthermore, if $a<p$ or $b<-p$, then the natural morphism induced by the same
  morphism as above is an injection:
  \begin{equation*}
    \xymatrix@R=.1em{%
      H^1(\sO_Y(a,b)\otimes \sym^2F^*\sB)
      \ar@{^(->}[r]  & H^1(\sO_Y(a,b)\otimes \sym^2F^*\sG).
    }
  \end{equation*}
\end{lem}

\begin{proof}
  Consider the Frobenius pull-back of the first vertical short exact sequence from
  (\ref{eq:6}):
  \[
  \xymatrix{%
    0\ar[r] & \sO_Y(-p,0)\ar[r] & F^*\sB \ar[r] & F^*\sG \ar[r] & 0
  }
  \]
  Since $n>2$, this, combined with \autoref{lem:a-b}, implies the first statement.

  Next, observe that this short exact sequence also implies that there exists a
  filtration
  \[
  \sym^2F^*\sB \supseteq \sE \supseteq \sO_Y(-2p,0)
  \]
  such that (after twisting by $\sO_Y(a,b)$) we have the short exact sequences
  \begin{equation}
    \label{eq:4}
    \xymatrix@C=1.5em{%
      0\ar[r] & \sO_Y(a-2p,b) \ar[r] & \sO_Y(a,b)\otimes\sE \ar[r] &
      \sO_Y(a-p,b)\otimes F^*\sG \ar[r] & 
      0  } 
  \end{equation}
  and
  \begin{equation}
    \label{eq:5}
    \xymatrix@C=1.5em{%
      0\ar[r] &  \sO_Y(a,b)\otimes\sE \ar[r] &
      \sO_Y(a,b)\otimes   \sym^2F^*\sB \ar[r]  &
      \sO_Y(a,b)\otimes\sym^2F^*\sG \ar[r] &  
      0  }
  \end{equation}
  Then by the first statement, \autoref{lem:B-coker}, \autoref{lem:a-b}, and
  (\ref{eq:4}) it follows that if $a<p$ or $b<-p$, then
  \[
  H^1(Y, \sO_Y(a,b)\otimes\sE )=0.
  \]
  By (\ref{eq:5}) this implies the second statement.
\end{proof}

\begin{lem}\label{lem:HFB-not-0}
  Let $a,b\in\bZ$ such that $a\geq 0$ and $b>-n$.  Then
  \[
  H^1(Y,\sO_Y(a,b)\otimes\sym^2F^*\sB)\neq 0 
  \]
\end{lem}

\begin{proof}
  Observe that the middle horizontal short exact sequence in the diagram (\ref{eq:6})
  implies that there exists a filtration
  \[
  \sym^2\left(F^*(V^\vee\otimes\sO_{Y})\right) \supseteq \sF \supseteq \sym^2F^*\sB
  \]
  such that (after twisting by $\sO_Y(a,b)$) we have the short exact sequences
  \begin{equation}
    \label{eq:13}
    \xymatrix@C=1.5em{%
      0\ar[r] & \sO_Y(a,b)\otimes\sym^2F^*\sB \ar[r] & \sO_Y(a,b)\otimes\sF \ar[r] &
      \sO_Y(a,b+p)\otimes F^*\sB \ar[r] & 
      0  } 
  \end{equation}
  and
  \begin{equation}
    \label{eq:12}
    \xymatrix@C=1.25em{%
      0\ar[r] &  \sO_Y(a,b)\otimes\sF \ar[r] &
      \sO_Y(a,b)\otimes   \sym^2\left(F^*(V^\vee\otimes\sO_{Y})\right) \ar[r]  &
      \sO_Y(a,b+2p) \ar[r] &  
      0.  } 
  \end{equation}
  Since $n>2$, it follows from (\ref{eq:12}) and \autoref{lem:a-b} that
  \begin{equation*}
    H^2(Y,\sO_Y(a,b)\otimes \sF)=0
  \end{equation*}
  and that 
  \begin{multline*}
    H^1(Y,\sO_Y(a,b)\otimes \sF)\simeq \\ \simeq \xymatrix{%
      \coker\left[ H^0\left(Y,\sO_Y(a,b)\otimes
          \sym^2\left(F^*(V^\vee\otimes\sO_Y)\right) \right) \ar[r]^-{\eta_2}
      \right. & \left. H^0(Y,\sO_Y(a,b+2p)) \right]. }
  \end{multline*}
  The morphism $\eta_2$ here is given by the matrix
  $[y_i^py_j^p \ver i,j=0,\dots,n]$.
  Furthermore, it follows from \autoref{lem:B-coker} that
  \begin{multline*}
    H^1(Y,\sO_Y(a,b+p)\otimes F^*\sB)\simeq \\ \simeq \xymatrix{%
      \coker\left[ H^0(Y,\sO_Y(a,b+p)\otimes F^*(V^\vee\otimes\sO_Y) ) \ar[r]^-{\eta_1}
      \right. & \left. H^0(Y,\sO_Y(a,b+2p)) \right]. }
  \end{multline*}
  The morphism $\eta_1$ here is given by the matrix $[y_i^p \ver i=0,\dots,n]$.
  Note, that by assumption $a\geq 0$ and $b\geq -n+1\geq -p$, so
  $H^0(Y,\sO_Y(a,b+p)\otimes F^*(V^\vee\otimes\sO_Y) )\neq 0$.
  Then it is easy to see, for example from the description of $\eta_1$ and $\eta_2$
  above, that
  \[
  \im \eta_2 \subsetneq \im \eta_1,
  \]
  and hence combined with (\ref{eq:13}) the above imply that
  \begin{equation}
    \label{eq:15}
    H^1(Y,\sO_Y(a,b)\otimes \sym^2 F^*\sB)\simeq 
    \factor{  \im \eta_1}{\im \eta_2}\neq 0.
    \qedhere
  \end{equation}
\end{proof}

\begin{rem}\label{rem:musing}
  Observe that the previous argument was the place where working in positive
  characteristic was crucial. The morphisms $\eta_1$ and $\eta_2$ are given by the
  $p^\text{th}$ powers of the global sections of $\sO_Y(0,1)$. We obtain the
  non-trivial cokernels and the ``gap'' between them from the fact that the global
  sections of $\sO_Y(0,p)$ are not generated by these $p^\text{th}$ powers.

  This argument fails for several reasons in characteristic $0$. First of all,
  $p^\text{th}$ powers do not define an $\sO_Y$-module homomorphism. Of course, they
  do not define one in any characteristic, which is the reason that we first have to
  pull-back everything by the Frobenius. However, the $p^\text{th}$ powers do give an
  $F^*\sO_Y$-module homomorphism. There is of course no Frobenius in characteristic
  $0$, but one might think that then one could use another finite morphism to
  pull-back these sections and thereby replacing the global sections by an
  appropriate power. However, in characteristic $0$ this would mean switching to an
  actual cover many of whose properties would change. For instance, very likely that
  cover would no longer be Fano or even have negative Kodaira dimension and other
  parts of the proof would break down.

  To summarize, the reason this argument works in positive characteristic is that
  there is a high degree endomorphism which is one-to-one on points. Then again, this
  is not surprising at all as this is usually the reason when a statement holds in
  positive characteristic but not in characteristic $0$.
\end{rem}

\begin{thm}\label{thm:Kod-fails}
  If $p\leq n=3$, then $\dim X=6$, and $H^{5}(X,\omega_X^{2})\neq 0$.
\end{thm}

\begin{proof}
  By Serre duality and (\ref{eq:1}) we have that
  \begin{multline}\label{eq:2}
    \begin{aligned}
      \qquad H^{i}(X,\omega_X^{2})^\vee 
      \simeq     H^{6-i}(X,\omega_X^{-1})  
      & \simeq  H^{6-i}(X,  \pi^*\sO_Y(3-p, 3-p)\otimes\sO_\pi(2))  \\
      & \simeq  H^{6-i}(Y,  \sO_Y(3-p, 3-p)\otimes\pi_*\sO_\pi(2))  \\
      & \simeq  H^{6-i}(Y,  \sO_Y(3-p, 3-p)\otimes \sym^2F^*\sG')  \\
      & \simeq  H^{6-i}(Y,  \sO_Y(3-p, 3+p)\otimes \sym^2F^*\sG)  \\
    \end{aligned}
  \end{multline}
  Since $p> a=3-p \geq$ and $b=3+p>-n=-3$, the statement follows from
  \autoref{lem:B-G} and \autoref{lem:HFB-not-0}.
\end{proof}

This might seem to give a desired example in $p=3$ as well, but this non-vanishing is
only interesting when $X$ is Fano, i.e., when $\omega_X^{-1}$ is ample and that only
holds when 
$p=2$.

\begin{cor}\label{cor:Fano-Kod-fails}
  If $n=3$ and $p=2$, then $X$ is a Fano variety on which $\omega_X^{-1}$ is very
  ample and $\omega_X^{-2}$ violates Kodaira vanishing. In particular,
  \autoref{thm:main1} follows.
\end{cor}

\begin{proof}
  If $n=3$ and $p=2$, then by (\ref{eq:1}) $\omega_X\simeq \pi^*\sO_Y(-1,-1)\otimes
  \sO_\pi(-2)$ and hence $\omega_X^{-1}$ is very ample by \autoref{say:very}. By
  \autoref{thm:Kod-fails}, $\omega_X^{-2}$ violates Kodaira vanishing.
\end{proof}

\begin{cor}
  \autoref{thm:main2} holds.
\end{cor}

\begin{proof}
  Let $Z=C_a(X,\omega_X^{-1})$. Then the statement follows from
  \autoref{cor:Fano-Kod-fails} and \autoref{prop:klt}.
\end{proof}


\begin{thebibliography}{DCF15}

\bibitem[Ber17]{Ber17}
{\sc F.~Bernasconi}: \emph{Kawamata-{V}iehweg vanishing fails for log del
  {P}ezzo surfaces in char $3$}. {\sf\scriptsize arXiv:1709.09238}

\bibitem[CT16a]{CT16b}
{\sc P.~Cascini and H.~Tanaka}: \emph{Purely log terminal threefolds with
  non-normal centres in characteristic two}, 2016. {\sf\scriptsize
  {arXiv:1607.08590}}

\bibitem[CT16b]{CT16}
{\sc P.~Cascini and H.~Tanaka}: \emph{Smooth rational surfaces violating
  {K}awamata-{V}iehweg vanishing}, 2016. {\sf\scriptsize {arXiv:1607.08542}}

\bibitem[DI87]{MR894379}
{\sc P.~Deligne and L.~Illusie}: \emph{Rel\`evements modulo {$p^2$}\,et
  d\'ecomposition du complexe de de {R}ham}, Invent. Math. \textbf{89} (1987),
  no.~2, 247--270. {\sf\scriptsize MR 894379}

\bibitem[DCF15]{dCF15}
{\sc G.~Di~Cerbo and A.~Fanelli}: \emph{Effective {M}atsusaka's theorem for
  surfaces in characteristic {$p$}}, Algebra Number Theory \textbf{9} (2015),
  no.~6, 1453--1475. {\sf\scriptsize MR 3397408}

\bibitem[Eke88]{MR972344}
{\sc T.~Ekedahl}: \emph{Canonical models of surfaces of general type in
  positive characteristic}, Inst. Hautes \'Etudes Sci. Publ. Math. (1988),
  no.~67, 97--144. {\sf\scriptsize MR 972344}

\bibitem[Elk81]{Elkik81}
{\sc R.~Elkik}: \emph{Rationalit\'e des singularit\'es canoniques}, Invent.
  Math. \textbf{64} (1981), no.~1, 1--6. {\sf\scriptsize MR 621766}

\bibitem[GR70]{Gra-Rie70b}
{\sc H.~Grauert and O.~Riemenschneider}: \emph{Verschwindungss\"atze f\"ur
  analytische {K}ohomologiegruppen auf komplexen {R}\"aumen}, Invent. Math.
  \textbf{11} (1970), 263--292. {\sf\scriptsize MR 0302938 (46 \#2081)}

\bibitem[HW17]{HW17}
{\sc C.~D. Hacon and J.~Witaszek}: \emph{On the rationality of {K}awamata log
  terminal singularities in positive characteristic}. {\sf\scriptsize
  arXiv:1706.03204}

\bibitem[Kol96]{Kollar96}
{\sc J.~Koll{\'a}r}: \emph{Rational curves on algebraic varieties}, Ergebnisse
  der Mathematik und ihrer Grenzgebiete. 3. Folge. A Series of Modern Surveys
  in Mathematics, vol.~32, Springer-Verlag, Berlin, 1996. {\sf\scriptsize MR
  1440180}

\bibitem[Kol13]{SingBook}
{\sc J.~Koll{\'a}r}: \emph{Singularities of the minimal model program},
  Cambridge Tracts in Mathematics, vol. 200, Cambridge University Press,
  Cambridge, 2013, with the collaboration of {\sc S{\'a}ndor J Kov{\'a}cs}.

\bibitem[KM98]{KM98}
{\sc J.~Koll{\'a}r and S.~Mori}: \emph{Birational geometry of algebraic
  varieties}, Cambridge Tracts in Mathematics, vol. 134, Cambridge University
  Press, Cambridge, 1998, With the collaboration of C. H. Clemens and A. Corti,
  Translated from the 1998 Japanese original. {\sf\scriptsize MR 1658959}

\bibitem[LR97]{LR97}
{\sc N.~Lauritzen and A.~P. Rao}: \emph{Elementary counterexamples to {K}odaira
  vanishing in prime characteristic}, Proc. Indian Acad. Sci. Math. Sci.
  \textbf{107} (1997), no.~1, 21--25. {\sf\scriptsize MR 1453823}

\bibitem[Lau96]{MR1385284}
{\sc N.~Lauritzen}: \emph{Embeddings of homogeneous spaces in prime
  characteristics}, Amer. J. Math. \textbf{118} (1996), no.~2, 377--387.
  {\sf\scriptsize MR 1385284}

\bibitem[Mad16]{MR3493588}
{\sc Z.~Maddock}: \emph{Regular del {P}ezzo surfaces with irregularity}, J.
  Algebraic Geom. \textbf{25} (2016), no.~3, 401--429. {\sf\scriptsize MR
  3493588}

\bibitem[Muk13]{MR3079312}
{\sc S.~Mukai}: \emph{Counterexamples to {K}odaira's vanishing and {Y}au's
  inequality in positive characteristics}, Kyoto J. Math. \textbf{53} (2013),
  no.~2, 515--532. {\sf\scriptsize MR 3079312}

\bibitem[Ray78]{MR541027}
{\sc M.~Raynaud}: \emph{Contre-exemple au ``vanishing theorem'' en
  caract\'eristique {$p>0$}}, C. {P}. {R}amanujam---a tribute, Tata Inst. Fund.
  Res. Studies in Math., vol.~8, Springer, Berlin, 1978, pp.~273--278.
  {\sf\scriptsize MR 541027}

\bibitem[Sch07]{MR2347424}
{\sc S.~Schr{\"o}er}: \emph{Weak del {P}ezzo surfaces with irregularity},
  Tohoku Math. J. (2) \textbf{59} (2007), no.~2, 293--322. {\sf\scriptsize MR
  2347424}

\bibitem[SB91]{MR1128214}
{\sc N.~I. Shepherd-Barron}: \emph{Unstable vector bundles and linear systems
  on surfaces in characteristic {$p$}}, Invent. Math. \textbf{106} (1991),
  no.~2, 243--262. {\sf\scriptsize MR 1128214}

\bibitem[SB97]{MR1440723}
{\sc N.~I. Shepherd-Barron}: \emph{Fano threefolds in positive characteristic},
  Compositio Math. \textbf{105} (1997), no.~3, 237--265. {\sf\scriptsize MR
  1440723}

\bibitem[Tot17]{Tot17}
{\sc B.~Totaro}: \emph{The failure of {K}odaira vanishing for {F}ano varieties,
  and terminal singularities that are not {C}ohen-{M}acaulay}. {\sf\scriptsize
  arXiv:1710.04364}

\bibitem[Yas14]{MR3230848}
{\sc T.~Yasuda}: \emph{The {$p$}-cyclic {M}c{K}ay correspondence via motivic
  integration}, Compos. Math. \textbf{150} (2014), no.~7, 1125--1168.
  {\sf\scriptsize MR 3230848}

\bibitem[Yas17]{Yas17}
{\sc T.~Yasuda}: \emph{Discrepancies of p-cyclic quotient varieties}.
  {\sf\scriptsize arXiv:1710.06044}

\end{thebibliography}

\def\cprime{$'$} \def\polhk#1{\setbox0=\hbox{#1}{\ooalign{\hidewidth
  \lower1.5ex\hbox{`}\hidewidth\crcr\unhbox0}}} \def\cprime{$'$}
  \def\cprime{$'$} \def\cprime{$'$} \def\cprime{$'$}
  \def\polhk#1{\setbox0=\hbox{#1}{\ooalign{\hidewidth
  \lower1.5ex\hbox{`}\hidewidth\crcr\unhbox0}}} \def\cdprime{$''$}
  \def\cprime{$'$} \def\cprime{$'$} \def\cprime{$'$} \def\cprime{$'$}
  \def\cprime{$'$}
\providecommand{\bysame}{\leavevmode\hbox to3em{\hrulefill}\thinspace}
\providecommand{\MR}{\relax\ifhmode\unskip\space\fi MR}
\providecommand{\MRhref}[2]{%
  \href{http://www.ams.org/mathscinet-getitem?mr=#1}{#2}
}
\providecommand{\href}[2]{#2}

\end{document}